\documentclass[12pt,a4paper,reqno]{amsart}

\usepackage{amsmath,amsthm,amssymb,latexsym,a4wide,tikz,multicol,tikz-cd}
\usepackage{arydshln,multirow}
\usepackage{tikz-qtree,tikz-qtree-compat}
\usepackage{mathtools,stmaryrd}
\usepackage{comment}
\usepackage{tikz}
\tikzset{font=\small}
\usepackage{enumerate}
\usetikzlibrary{matrix,arrows}
\usetikzlibrary{positioning}
\usetikzlibrary{cd}

\newtheorem{theorem}{Theorem} [section]
\newtheorem{lemma}[theorem]{Lemma}
\newtheorem{corollary}[theorem]{Corollary}
\newtheorem{proposition}[theorem]{Proposition}

\theoremstyle{definition}

\newtheorem{problem}[theorem]{Problem}
\newtheorem{example}[theorem]{Example}

\newtheorem{remark}[theorem]{Remark}

\newtheorem{definitions}[theorem]{Definitions}


\usepackage[active]{srcltx}

\usepackage{enumerate}
\numberwithin{equation}{section}

\author{Nasir Sohail, Aftab Hussain Shah, Kristo Väljako}   
\address{Institute of Mathematics and Statistics, University of Tartu,
Tartu, Estonia}
\email{nasir.sohail@ut.ee} 
\address{Department of Mathematics, Central University of Kashmir, Ganderbal, Jammu \& Kashmir, India }
\email{aftab@cukashmir.ac.in}

\address{Institute of Mathematics and Statistics, University of Tartu,
Tartu, Estonia \newline
\indent Institute of Computer Science, University of Tartu,
Tartu, Estonia}

\email{kristo.valjako@ut.ee}
\title{Embedding ample semigroups as $(2,1,1)$-subalgebras of inverse semigroups}

\thanks{This research is being supported by the Estonian Research Council grant PRG1204.}

\begin{document}

\begin{abstract}  
In general, the problem of embedding an ample semigroup in an inverse semigroup as a $(2,1,1)$-type subalgebra is known to be undecidable. An affirmative answer to the problem is provided for certain classes of finite ample semigroups. We also give examples of semigroups that are left (respectively, right) but not right (respectively, left) ample. 
\end{abstract}

\maketitle

\section{Motivation and preliminaries}
Let $X$ be a (possibly empty) set. By a \textit{partial bijection} of $X$ we mean a bijection $\beta: A\longrightarrow B$ such that $A$ and $B$, called, respectively, the \textit{domain} and \textit{image} of $\beta$, are subsets of $X$---the empty map is a partial bijection having empty domain and image. We shall denote $A$ and $B$ by $Dom\beta$ and $Im\beta$, respectively. The set of all partial bijections of $X$ is denoted by $\mathcal{I}_X$.\vspace{2pt}

Throughout this article, the maps will be written to the right of their arguments. Accordingly, $\phi \circ \psi$, viz. a composition of full maps, will mean \lq $\phi$ followed by $\psi$'. Also, for better readability, parentheses around the arguments will be sometimes dropped if there is no chance of ambiguity. For all $\beta, \gamma \in \mathcal{I}_X$, we define
\begin{equation}\label{Comp1}
\beta \cdot \gamma = \beta|_{(Im\beta \cap Dom \gamma)\beta^{-1}} \circ \gamma \in \mathcal{I}_X,
\end{equation}
where $\beta^{-1}:Im \beta \longrightarrow Dom \beta$ denotes the inverse of $\beta$. The binary operation (\ref{Comp1}) turns $\mathcal{I}_X$ into a semigroup, called the \textit{symmetric inverse semigroup over $X$}. As usual, we shall simply write $\beta \gamma$ instead of $\beta \cdot \gamma$. The  \textit{empty map}, which will be denoted by $0$, acts as the zero element of $\mathcal{I}_X$. Actually, $\mathcal{I}_X$ is a monoid whose identity is the (full) identity map $1_X$ on $X$.  Note that for every $\beta \in \mathcal{I}_X$ there exists a unique partial bijection $\beta^{-1} \in \mathcal{I}_X$, namely the inverse of $\beta$, such that $\beta \beta^{-1}$ and $\beta^{-1} \beta$ are identities on $Dom \beta$ and $Im\beta$, respectively. Particularly, $\beta \beta^{-1}\beta= \beta$ and $\beta^{-1}\beta\beta^{-1}= \beta^{-1}$.\vspace{2pt}

In general, an element $a$ of a semigroup $S$ is said to be \textit{invertible} if there exists a unique element $a^{-1} \in S$, called the \textit{inverse} of $a$,  such that $aa^{-1}a=a$ and $a^{-1}aa^{-1}=a^{-1}$. We call $S$ an \textit{inverse semigroup} if all of its elements are invertible. The semigroup $\mathcal{I}_X$, discussed above, is a typical example of an inverse semigroup. In what follows, the letter $T$, possibly with subscripts, will be reserved to only denote inverse semigroups.\vspace{2pt}

A \textit{morphism} $h:T_1 \longrightarrow T_2$ of inverse semigroups is just a semigroup homomorphism from $T_1$ to $T_2$. \textit{Monomorphisms} and \textit{isomorphisms} of (inverse) semigroups are, respectively, the injective and bijective homomorphisms.   The Wagner-Preston representation (see Theorem \ref{Wagnr-Prstn} below) asserts that every inverse semigroup can be embedded in a symmetric inverse semigroup.

\begin{theorem}\label{Wagnr-Prstn}
    Let $T$ be an inverse semigroup. Then, for all $x\in T$ the map
    \[
    \rho_x:Txx^{-1}\longrightarrow Tx^{-1}x,
    \]
    givne by
    \[
    (a)\rho_x = ax, \,\, \forall\, a \in Txx^{-1},
    \]
is a partial bijection of $T$. Furthermore, the map
\[
\rho:T\longrightarrow \mathcal{I}_T, \text{ given by }\, x \longmapsto \rho_x,\,\, \forall\, x \in T,
\]
is a monomorphism.
\end{theorem}

\begin{proof}
    See, for instance, \cite{Howie} Theorem 5.1.7.
\end{proof}

\begin{theorem}\label{Wagnr-Prstn_D}
    Given an inverse semigroup $T$, the map
    \[
    {\lambda_x}:x^{-1}xT\longrightarrow xx^{-1}T,
    \]
    defined by
    \[
    (a){\lambda_x} = xa,\,\, \forall \, a \in x^{-1}xT,
    \]
is a partial bijection of $T$ for all $x\in T$. Also, the map
\[
\lambda:T\longrightarrow \mathcal{I}_T, \text{ defined as }\, x \longmapsto {\lambda_x},\,\,\forall\, x\in T,
\]
is an anti-monomorphism; i.e. $\lambda$ is injective and $(xy)\lambda = (y\lambda)(x\lambda)$, for all $x,y \in T$.
\end{theorem}

\begin{proof}
    Compare, for instance, with \cite{Lawson1998} Theorem 1.5.1---the \lq isomorphism' is replaced with \lq anti-isomorphism' because, contrary to \cite{Lawson1998}, we are writing maps to the right of their arguments.
\end{proof}
 
\noindent Clearly, the map $\vartheta:(T)\rho \longrightarrow (T)\lambda$, defined by $\rho_x \longmapsto {\lambda_x}$, for all $x \in T$, is an anti-isomorphism, where $\rho$ and $\lambda$ are maps given by Theorems \ref{Wagnr-Prstn} and \ref{Wagnr-Prstn_D}, respectively. One can also easily observe that
\begin{equation*}\label{DomsandImgs}
    Txx^{-1} = Tx^{-1},\; Tx^{-1}x = Tx,\; x^{-1}xT = x^{-1}T, \text{ and }\; xx^{-1}T= xT, \, \forall \, x \in T.
\end{equation*}

As usual, the set of idempotents of a semigroup $S$ will be denoted by $E(S)$. For an inverse semigroup $T$, we have
\[
E(T) =\{xx^{-1}: x \in T\}.
\]
Actually, $E(T)$ is a subsemilattice of $T$. Identifying $T$ with its isomorphic copy in $\mathcal{I}_T$ we see that $xx^{-1}, x^{-1}x\in E(T)$ are, respectively, the identities on the domain and image of $x\in T$. This fact may be used to verify that the conditoin
\begin{equation}\label{Eq1.1}
    \forall\, x,y\in T,\, x \leq y \text{ iff } x=xx^{-1}y
\end{equation}
defines a partial order on $T$---one may interpret (\ref{Eq1.1}) as $x \leq y$ if and only if the domain of $x$ is a subset of the domain of $y$ and the latter agrees with the former on the restricted domain. It is an easy exercise to show that Condition (\ref{Eq1.1}) may equivalently be stated as:
\begin{equation*}\label{Eq1.2}
    \forall\, x,y\in T,\, x \leq y \text{ iff }\, \exists\, e\in E(T), \text{ such that }x=ey.
\end{equation*}
 We call $\leq$ the \textit{natural partial order} on $T$. In fact, for any semigroup $S$ the set $E(S)$ comes equipped with the partial order:
\begin{equation}\label{Eq1.0}
    e \preccurlyeq f \,\text{ iff }\, ef = fe = e,\; \forall\, e,f \in E(S).
\end{equation}
The natural partial order of an inverse semigroup $T$ is an extension of the partial order $\preccurlyeq$ on $E(T)$. In the sequel, \textit{order} on an inverse semigroup will mean the natural partial order, as defined by (\ref{Eq1.1}).\vspace{2pt}

Given an inverse semigroup $T$, let us further recall (\cite{NS}, Remark 4.3) that for all $x\in T$, $xx^{-1}$ and $x^{-1}x$ are the minimum idempotents such that $x=xx^{-1}x$, i.e. for all $e,f \in E(T)$, $ex = x$ implies that $xx^{-1} \leq e$ and $xf = x$ implies that $x^{-1}x \leq f$.\vspace{2pt}

A semigroup $S$ is called \textit{left ample} if it can be embedded in an inverse semigroup $T$ such that $(x\phi)(x\phi)^{-1} \in S\phi$ for all $x\in S$, where $\phi$ is the embedding of $S$ into $T$. We then denote
\begin{equation}\label{Identification}
xx^{-1} := [(x\phi')(x\phi')^{-1}]\phi'^{-1} \in S, \, \forall \, x \in S,
\end{equation}
where $\phi':S\longrightarrow S\phi$ is the isomorphism defined by $x\longmapsto x\phi$. Thus, a left ample semigroup comes equipped with the unary operation $x \longmapsto x^+ := xx^{-1}$. Similarly, a semigroup $S$ is called \textit{right ample} if there exists an inverse semigroup $T$ admitting a monomorphism $\phi:S \longrightarrow T$ such that $(x\phi)^{-1}(x\phi) \in S\phi$ for all $x \in S$. We use the notation
\[
x^{-1}x := \left[(x\phi')^{-1}(x\phi')\right]\phi'^{-1} \in S
\]
for all $x\in S$, where $\phi'$ is the isomorphism considered above. A right ample semigroup is endowed with the unary operation $x \longmapsto x^* := x^{-1}x$. We shall refer to $T$ as an \textit{associated inverse semigroup} of the (left, right) ample semigroup $S$. If $S$ is a subsemigroup of $T$ then we shall say that $S$ is (\textit{left}, \textit{right}) \textit{ample in $T$}. One may easily verify that $S$ is left (respectively, right) ample in $T$ if and only if $(S)\rho$ is left (respectively, right) ample in $\mathcal{I}_T$, where $\rho$ is the monomorphisms given by Theorem \ref{Wagnr-Prstn}.\vspace{2pt}

A semigroup $S$ is called \textit{ample} if it has (possibly different) associated inverse semigroups $T_1$ and $T_2$ making it, respectively, into a left and a right ample semigroup. We shall say that $S$ is \textit{ample in} $T$ if it is both left and right ample in $T$.  Obviously, every inverse semigroup is ample (in itself). Also, a subsemigroup $S$ of an inverse semigroup $T$ is ample in $T$ if $E(T) \subseteq S$, that is, if $S$ is \textit{full} in $T$. The converse is not true; for example, $\mathbb{N}$ is ample but not full in the multiplicative monoid $\mathbb{Q}$. For more examples of (left, right) ample semigroups the reader is referred to \cite{Gould} and \cite{NS}, and the references contained therein. For a general background on inverse semigroups one may refer to \cite{Howie} or \cite{Lawson1998},

\vspace{2pt}
From the universal algebraic perspective, a left ample semigroup may be considered as a $(2,1)$-type algebra $(S,\cdot, ^+)$. Similarly, a right ample semigroup can be viewed as a $(2,1)$-type algebra $(S,\cdot, ^*)$. An ample semigroup may then be treated as a $(2,1,1)$-type algebra $(S, \cdot, ^+, ^*)$. Particularly, every inverse semigroup is an algebra of the latter type; and by saying that $S$ is left (right) ample in $T$ we precisely mean that $S$ is $(2,1)$-type subalgebra of $T$. Consequently, $S$ is ample in $T$ if and only if it is a $(2,1,1)$-type subalgebra of $T$.

\begin{problem}[\textbf{The embedding problem}]\label{Prblm}
Let $T_1$ and $T_2$ be different inverse semigroups containing isomorphic copies, say $S_1$ and $S_2$, of a non-inverse semigroup $S$. Let $S_1$ be left ample in $T_1$ and $S_2$ be right ample in $T_2$ (whence $S$, being an ample semigroup, is a $(2,1,1)$-type algebra).  Then, can we find an inverse semigroup $T$ containing an isomorphic copy of $S$ that is left as well as right ample in $T$? In terms of universal algebra, this amounts to embedding $S$ in an inverse semigroup as a $(2,1,1)$-type subalgebra.
\end{problem}

\noindent The following results imply that, in general, Problem \ref{Prblm} is undecidable.

\begin{theorem}[\cite{Gould and Kambites}, Theorem 3.4]\label{GK1}
    Let $S$ be an ample semigroup. Then $S$ is a $(2,1,1)$-type subalgebra of an inverse semigroup if and only if $S$ is a full subsemigroup of an inverse semigroup.
\end{theorem}

\begin{corollary}[\cite{Gould and Kambites}, Corollary 4.3]\label{GK2}
     It is undecidable whether a finite ample semigroup embeds as a full subsemigroup of a finite inverse semigroup, or of an inverse semigroup.
\end{corollary}

\noindent In the next section we answer Problem \ref{Prblm} for certain classes of finite semigroups. Particularly, we prove that a finite ample semigroups $S$ with central idempotents is embeddable in $\mathcal{I}_S$ as a $(2,1,1)$-subalgebra. Then, in the last section, we provide examples of semigroups that are left (respectively, right) ample but not right (respectively, left) ample. (To the knowledge of the authors, no such example exists in the literature.)

\begin{remark}\label{Rmrk1.4}
    Let a semigroup $S$ be left (respectively, right) ample in an inverse semigroup $T_1$, and $T_2$ be an inverse semigroup admitting a homomorphism $f: T_1 \longrightarrow T_2$. Then one can easily verify that $(S)f$ is left (respectively, right) ample in $T_2$.
\end{remark}

\section{Answering the embedding problem for some finite ample semigroups}

 Let $S$ be a subsemigroup of an inverse semigroup $T$. Following \cite{Schein}, a subset $Y$ of $T$ will be called \textit{right $S$-invariant} if
\[
    Im(\rho_s|_{Dom\rho_s\cap Y}) \subseteq Y, \hspace{4pt} \text{for all } s \in S,
\]
where $\rho_s:Tss^{-1} \longrightarrow Ts^{-1}s$ is the partial bijection defined in Theorem \ref{Wagnr-Prstn}.
Denoting $\rho_s|_{Dom\rho_s\cap Y}$, equivalently $\rho_s|_{Tss^{-1}\cap Y}$, by $\sigma^{Y}_s$, we introduce a map:
\begin{equation}\label{Eq2.1}
\sigma^{Y}:S\longrightarrow \mathcal{I}_Y \text{ by } (s)\sigma^{_Y} =\sigma^{Y}_s.
\end{equation}
Clearly, $\sigma^{Y}$ is well-defined. Also, it is known to be a homomorphism when $S$ is finite; see \cite{Schein}. The following lemma shows that the finiteness condition is, nevertheless, not necessary.

\begin{lemma}\label{Lema2.1}
Let $Y$ be a right $S$-invarient subset of $T$. Then, with the notations defined above, $\sigma^{Y}$ is a homomorphism.
\end{lemma}
\begin{proof}
We need to show that $\sigma^{Y}_s  \sigma^{Y}_t = \sigma^{Y}_{st}$, for all $s,t \in S$. By Theorem \ref{Wagnr-Prstn} we have $\rho_{st} = \rho_s \rho_t$, whence
\[
Dom\rho_{st} = Dom(\rho_s  \rho_t) = (Ts^{-1}s \cap Ttt^{-1})\rho^{-1}_s.
\]
Consequently, we have
\begin{equation*}\label{Hat1}
\sigma^Y_{st} = \rho_{st}|_{(Ts^{-1}s \cap Ttt^{-1})\rho^{-1}_s\cap Y} = (\rho_s  \rho_t)|_{(Ts^{-1}s \cap Ttt^{-1})\rho^{-1}_s\cap Y}.
\end{equation*}
On the other hand, one has
\[
\sigma^Y_s  \sigma^Y_t = \left(\rho_s|_{Tss^{-1}\cap Y}\right)\left( \rho_t|_{Ttt^{-1}\cap Y}\right).
\]
Particularly, we note that
\begin{equation*}
\begin{split}
Dom(\sigma^Y_{st}) &= (Ts^{-1}s \cap Ttt^{-1})\rho^{-1}_s\cap Y \\
Dom (\sigma^Y_s  \sigma^Y_t) &= [(Tss^{-1} \cap Y)\rho_s \cap (Ttt^{-1} \cap Y)]\rho_s^{-1}.
\end{split}
\end{equation*}
Now, using the fact that $\sigma^Y_x$ is a restrictions of $\rho_x$, for all $x \in S$, we may write for all $x\in Dom(\sigma^Y_{st}) \cap Dom(\sigma^Y_s \sigma^Y_t)$:
\[
(x)\sigma^Y_{st}  =(x)\rho_{st} =(x)(\rho_s  \rho_t) = (x)(\sigma^Y_s  \sigma^Y_t).
\]
Thus, the proof will be accomplished if we show that $Dom(\sigma^Y_s \sigma^Y_t) = Dom(\sigma^Y_{st})$. 
To this end, observe that
\begin{equation*}
\begin{split}
     x & \in [(Tss^{-1} \cap Y)\rho_s \cap (Ttt^{-1} \cap Y)]\rho^{-1}_s \\
    \iff x \rho_s & \in (Tss^{-1} \cap Y)\rho_s \cap Ttt^{-1} \cap Y \\
    \iff \hspace{10pt}  x & \in  Tss^{-1} \cap Y,\,\, x\rho_s \in Ttt^{-1} \cap Y\\
    \iff  \hspace{10pt} x &\in Tss^{-1},\,\, x \in Y, \,\, x\rho_s \in Ttt^{-1},\,\, x\rho_s \in Y\\
    \iff  x\rho_s &\in Ts^{-1}s,\,\, x\rho_s \in Ttt^{-1},\,\, x \in Y,\,\, x\rho_s \in Y\\
    \iff  x\rho_s &\in (Ts^{-1}s \cap Ttt^{-1}),\,\, x \in Y,\,\, x\rho_s \in Y\\
    \iff \hspace{10pt} x & \in (Ts^{-1}s \cap Ttt^{-1}) \rho^{-1}_s \cap Y,\\
\end{split}
\end{equation*}
where the ($\impliedby$) part of the last implication follows from the observation:
\[
x \in (Ts^{-1}s \cap Ttt^{-1}) \rho^{-1}_s \implies x\rho_s \in Ts^{-1}s \implies x \in Tss^{-1} \implies x \in Dom\rho_s.
\]
Hence $Dom(\sigma^Y_s \sigma^Y_t) = Dom(\sigma^Y_{st})$, and the proof is complete.
\end{proof}

\begin{proposition}\label{Prop2.2a}
    Let $S$ be a subsemigroup of an inverse semigroup $T$. Then the map $\sigma^{_S}: S \longrightarrow \mathcal{I}_S$, given by (\ref{Eq2.1}), is a homomorphism.
\end{proposition}

\begin{proof}
Follows from the fact that every subsemigroup $S$ of an inverse semigroup $T$ is a right $S$-invariant subset of $T$.
\end{proof}
\noindent In what follows, dropping the superscript, we shall simply write $\sigma$ instead of $\sigma^{_S}$. Although the first part of the following proposition also follows from \cite{Gould_Notes} Theorem 6.2 (iv), our essentially direct scheme of proof will be instructive in the sequel.

\begin{proposition}\label{Prop2.2}
If $S$ is left ample in $T$, then $\sigma: S \longrightarrow \mathcal{I}_S$, as given by Proposition \ref{Prop2.2a}, is a monomorphism. Also, for every $s\in S$ the domain and image of $(s)\sigma = \sigma_s$ are the sets $Sss^{-1}$ and $Ss$, respectively. Moreover, $(S)\sigma$ is left ample in $\mathcal{I}_S$.
\end{proposition}

\begin{proof}
To prove the first part of the proposition, we only need to show that $\sigma$ is injective. Assume, for this purpose, that
\[
\sigma_s = (s)\sigma = (t)\sigma = \sigma_t,\; \text{i.e. } \,\rho_s|_{Tss^{-1}\cap S} = \rho_t|_{Ttt^{-1}\cap S},
\]
for some $s,t \in S$. Because $S$ is left ample, we have $ss^{-1} \in Tss^{-1} \cap S$ and $tt^{-1} \in Ttt^{-1} \cap S$. Consequently, we may claim, using $Tss^{-1} \cap S = Ttt^{-1} \cap S$ (viz. the equality of domains of $\sigma_s$ and $\sigma_t$), that
\[
(ss^{-1})\rho_s = (ss^{-1})\rho_t\; \text{ and }\; (tt^{-1})\rho_s = (tt^{-1})\rho_t.
\]
This gives
\[
s = ss^{-1}s = ss^{-1}t \; \text{ and }\; t = tt^{-1}t = tt^{-1}s. 
\]
Now, it follows from the definition of order on $T$ that $s \leq t \leq s$. Thus we get $s = t$, whence $\sigma$ is injective.\vspace{2pt}

\noindent As for the second part of the proposition, let us first notice that
\[
Sss^{-1} \subseteq Tss^{-1} \cap S,\ \forall\, s\in S.
\]
Then observe that, for all $x \in T,\, s \in S$,
\[
x(ss^{-1}) \in Tss^{-1} \cap S \implies x(ss^{-1}) = x(ss^{-1})(ss^{-1}) \in Sss^{-1}.
\]
Hence, for every $s\in S$ one has
\begin{equation}\label{Dom}
    Sss^{-1}= Tss^{-1} \cap S.
\end{equation}
This implies that for every $s \in S$, the domain and image of $\sigma_s$ are, respectively, the sets
\begin{equation*}\label{Dom&Im}
Sss^{-1} \text{ and } Sss^{-1}s = Ss\,\, (\subseteq Ss^{-1}s \cap S).
\end{equation*}

To prove the last part of the proposition, we first consider the isomorphism $\rho': T \rho \longrightarrow T$ defined by $\rho_x \longmapsto x$ for all $x \in T$. Then, one may define an isomorphism
\[
\rho'|_{S\rho} \circ \sigma :S\rho \longrightarrow S\sigma
\]
given by
\[
\rho_s \longmapsto 
\sigma_s,\, \forall \, s\in S.
\]
Next, because $\sigma_s \sigma_s^{-1}$ is the identity on $Sss^{-1} = Tss^{-1} \cap S$ and $\rho_{ss^{-1}}$ is the identity on $Tss^{-1}$,
we have
\[
\sigma_s \sigma_s^{-1} = \rho_{ss^{-1}}|_{Dom\rho_{ss^{-1}}\cap S}.
\]
Now, since $ss^{-1}\in Dom\rho_{ss^{-1}} \cap S$, it follows from the above equality that
\[
\sigma_s \sigma_s^{-1} = (\rho_{ss^{-1}})\left(\rho' \circ \sigma \right) = \left( (\rho_{ss^{-1}}) \rho' \right) \sigma\in S\sigma.
\]

\noindent implying that $S\sigma$ is left ample in $\mathcal{I}_S$. 
\end{proof}

\begin{corollary} \label{Clry2.3}
    Let $S$ be left ample in $T$.
Also, let $\sigma_s = (s)\sigma$ and $\rho_s = (s)\rho$ be as considered in Proposition \ref{Prop2.2}. Then, considering the isomorphism
    \[
    \theta : (S)\rho \longrightarrow \mathcal{I}_S, \text{ defined by }\ \rho_s \longmapsto \sigma_s, \forall\, s \in S
    \]
from Proposition \ref{Prop2.2}, we have, for all $s\in S$,
    \[
    (\rho_s\rho_s^{-1})\theta = \sigma_s\sigma_s^{-1} = \sigma_{ss^{-1}},
    \]
    where $ss^{-1} = (\rho_s\rho_s^{-1})\rho^{-1} \in S$, cf. (\ref{Identification}).
\end{corollary}

\begin{proof}
Clearly, we have
\begin{equation*}    
(\rho_s\rho_s^{-1})\theta = (\rho_s\rho_{s^{-1}})\theta  =(\rho_{ss^{-1}})\theta 
= \sigma_{ss^{-1}}.
\end{equation*}
Also, using (\ref{Dom}), we can calculate:
\begin{equation*}
\sigma_{ss^{-1}} 
= \rho_{ss^{-1}}|_{Sss^{-1}}
= (\rho_s \rho_{s^{-1}})|_{Sss^{-1}}
= (\rho_s\rho_s^{-1})|_{Sss^{-1}}
= (\rho_s|_{Sss^{-1}})(\rho_s|_{Sss^{-1}})^{-1}
= \sigma_s\sigma_s^{-1}.
\end{equation*}
This completes the proof.
\end{proof}

\noindent With $S$ replaced with  $S_1$, Figure 1 shows all maps used in Proposition \ref{Prop2.2}, Corollary \ref{Clry2.3}.

\begin{remark}\label{RMRK2.5}
For any left ample semigroup $S$, it follows from the above corollary that
\[
\left(\sigma_s\sigma_s^{-1}\right)\sigma' = ss^{-1} = \left(\rho_s\rho_s^{-1}\right)\rho' ,\, \forall \, s \in S,
\]
with $\rho'$ introduced in the proof of Proposition \ref{Prop2.2} and $\sigma'$ being defined analogously.
\end{remark}

In Proposition \ref{Prop2.2} and its corollary, $S$ was left ample in $T$. In the following theorem $S$ is not necessarily contained in its associated inverse semigroup. (In fact, this result relates more closely to Theorem 6.2 (iv) of \cite{Gould_Notes}; cf. the paragraph before Proposition \ref{Prop2.2}.)

\begin{theorem}\label{Prop2.4}
Let $S$ be a left ample semigroup. Then, for every $x\in S$ the mapping $\hat{\rho}_x:Sxx^{-1}\longrightarrow Sx$ defined by $(z)\hat{\rho}_x = zx$, for all $z \in Sxx^{-1}$, is a bijection. Furthermore, the function $\hat{\rho}:S \longrightarrow \mathcal{I}_S$ given by $x\longmapsto \hat{\rho}_x$ is a monomorphism, and $(S)\hat{\rho}$ is left ample in $\mathcal{I}_S$.
\end{theorem}

\begin{proof}
To help the reader, the morphisms used in the proof are shown in Figure 1.
Let $T$ be an (associated) inverse semigroup admitting a monomorphism
\[\phi:S\longrightarrow T,
\]
such that $(S)\phi = S_1$ is left ample in $T$. Let $\phi':S \longrightarrow S_1$ be the isomorphism defined by $x \longmapsto (x)\phi$, and $\sigma:S_1 \longrightarrow \mathcal{I}_{S_1}$, defined by
\[
(x)\phi' \longmapsto \sigma_{(x)\phi'},\, \forall\, x \in S,
\]
be the monomorphism given by Proposition \ref{Prop2.2}. Particularly, $S_1\sigma$ is left ample in $\mathcal{I}_{S_1}$.\vspace{2pt}

Let $\hat{\phi} = \phi' \circ \sigma:S \longrightarrow \mathcal{I}_{S_1}$, and $x$ be an arbitrary element of $S$. Then, defining $s = (x)\phi'$, and using Corollary \ref{Clry2.3}, we may calculate,
\begin{equation*}\label{Obtain}
\begin{split}
(x\hat{\phi})
(x\hat{\phi})^{-1} 
&= \left(\left(x\phi'\right)\sigma \right)
\left(\left(x\phi'\right)\sigma \right)^{-1} \\
&= \left( s\sigma \right)\left( s\sigma \right)^{-1}\\
&=\sigma_s\sigma_s^{-1} \\
&=\sigma_{ss^{-1}}\\
&=\sigma_{(x\phi')(x\phi')^{-1}}\\
&= \sigma_{(xx^{-1})\phi'}, \hspace{14pt} \text{refer to (\ref{Identification})}, \\
&= (xx^{-1})\hat{\phi}.
\end{split}
\end{equation*}
Thus, we may write \vspace{2pt}
\begin{equation}\label{xxINV}
[(x\hat{\phi})(x\hat{\phi})^{-1}]\hat{\phi}^{-1} = [(x\phi')(x\phi')^{-1}]\phi'^{-1} =  xx^{-1} \in S.\vspace{2pt}
\end{equation}

\begin{figure}[htbp]
\centering
\includegraphics [width=90mm]{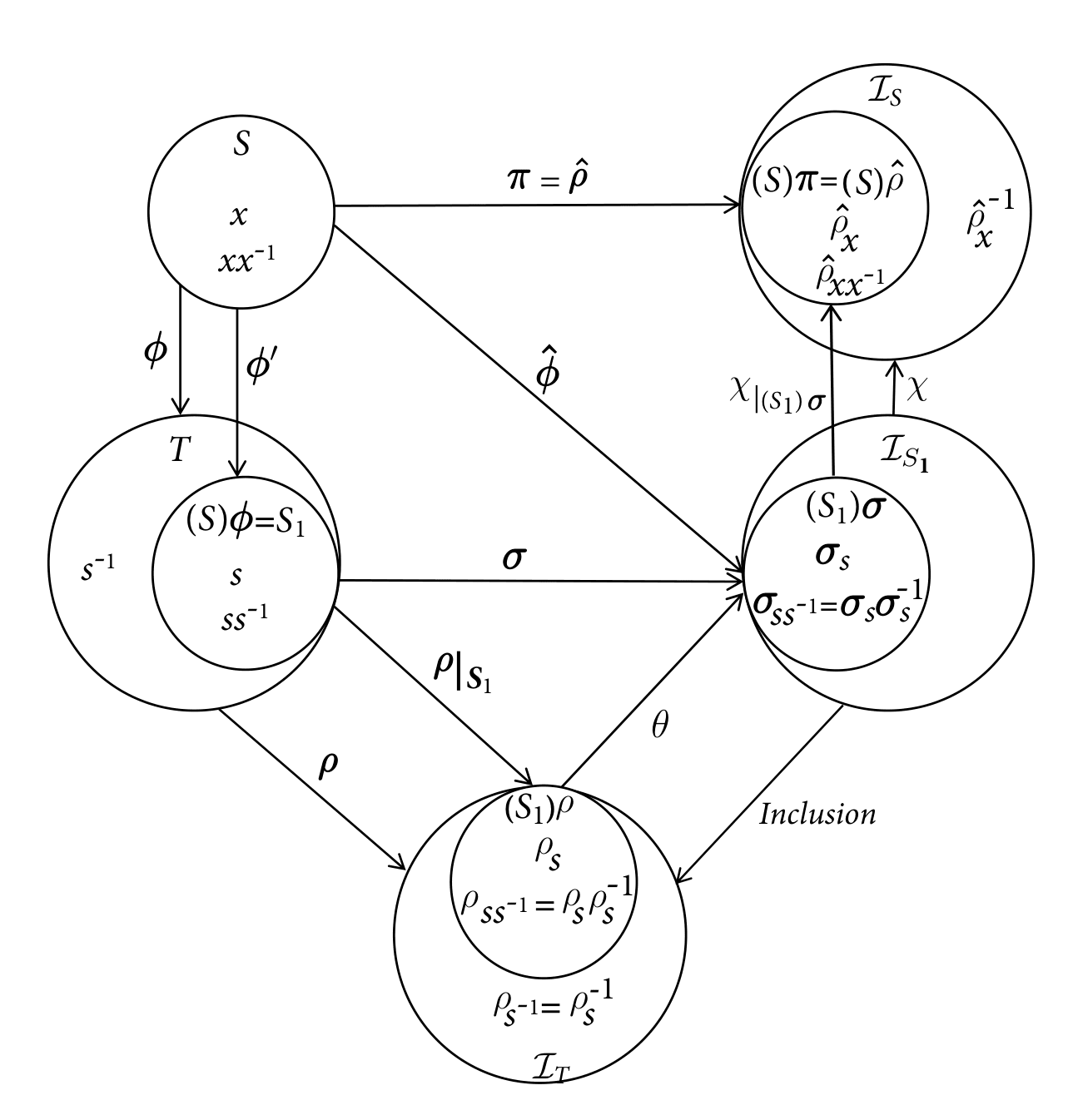}
\caption{}
\end{figure}\label{Diagram1}

\noindent Now, for an arbitrary partial bijection $\beta : A \longrightarrow B$ in $\mathcal{I}_{S_1}$, consider the map
\[
\gamma:(A)(\phi')^{-1} \longrightarrow (B)(\phi')^{-1}
\]
defined by
\begin{equation}\label{map2.1}
(a)(\phi')^{-1} \longmapsto (a\beta)(\phi')^{-1}, \, \forall\, a \in A.
\end{equation}
Clearly, $\gamma$ is an element of $\mathcal{I}_S$, and it is a routine to verify that $\chi: \mathcal{I}_{S_1} \longrightarrow \mathcal{I}_S$ given by $\beta \longmapsto \gamma$ is an isomorphism. Moreover, by Remark \ref{Rmrk1.4}, $(S\hat{\phi})\chi = (S_1\sigma)\chi$ is left ample in $\mathcal{I}_{S}$.\vspace{2pt}

To complete the proof, let us consider the monomorphism $\pi = \hat{\phi} \circ {\chi}\colon S \longrightarrow \mathcal{I}_S$.
The aim is to show that $\pi = \hat{\rho}$, i.e. for all $x \in S$, $(x)\pi = \hat{\rho}_x$
where $\hat{\rho}_x$ is as given in the statement of the theorem. Recall from Proposition \ref{Prop2.2} that the domain and image of $\sigma_{(x)\phi}$ ($=\sigma_s$) are, respectively, $(S_1)(x\phi')(x\phi')^{-1}$ and $(S_1)(x\phi')$. Hence, by the definitions of $\pi$ and $\chi$, we have
\[
(x)\pi = \left( \sigma_{(x)\phi}\right)\chi: \{S_1(x\phi')(x\phi')^{-1}\} (\phi')^{-1}\longrightarrow \{S_1(x\phi')\}(\phi')^{-1},\vspace{2pt}
\]
defined according to (\ref{map2.1}). However, it is immediate that
\[
\{S_1(x\phi')\}(\phi')^{-1} = Sx,
\]
and by (\ref{xxINV}), we have
\[
\{S_1(x\phi')(x\phi')^{-1}\} (\phi')^{-1} = Sxx^{-1}.
\]

\noindent Also, taking $z =uxx^{-1}, u \in S$, we may calculate, using (\ref{map2.1}),
\begin{equation*}
\begin{split}
 (z)(x\pi) &= (uxx^{-1}) [(x)\hat{\phi}\circ \chi]\\
&=\left(uxx^{-1}\right)\left(\left(\sigma_{(x)\phi'}\right)\chi\right)\\
&=((uxx^{-1})\phi'(x)\phi')\phi'^{-1}\\
\end{split}
\end{equation*}

\begin{equation*}
\begin{split}
&=((uxx^{-1}x)\phi')\phi'^{-1}\\
&= uxx^{-1}x\\
&=zx.
\end{split}
\end{equation*}
So, we conclude that
\[
(x)\pi: Sxx^{-1} \longrightarrow Sx, 
\]
is defined by $(z)(x\pi) =zx$ for all $z =uxx^{-1} \in Sxx^{-1}$ ($u\in S$), hence the proof.
\end{proof}

Before stating our main result, let us note that using the dual notion of \textit{left $S$-invariance}, one can similarly prove that any right ample subsemigroup $S$ of an inverse semigroup $T$ has an anti-isomorphic copy in the symmetric inverse semigroup $\mathcal{I}_S$ given by the anti-monomorphism,
\begin{equation*}
\delta:S \longrightarrow \mathcal{I}_S, \hspace{5pt} \text{given by } s \longmapsto {\delta_s} =  {\lambda_s}|_{s^{-1}sT \, \cap\, S},\, \forall\, s \in S,
\end{equation*}
where $\lambda_s \in \mathcal{I}_T$ is the partial bijection defined in Theorem \ref{Wagnr-Prstn_D}. Moreover, $s^{-1}sT \cap S = s^{-1}sS$, and we have the following dual of Theorem \ref{Prop2.4}.

\begin{theorem}\label{Prop2.4a}
For any right ample semigroup $S$, there exists an anti-monomorphism,
\[
\hat{\lambda}:S \longrightarrow \mathcal{I}_S, \text{ given by } x \longmapsto \hat{\lambda}_x,
\]
where the partial bijection,
\begin{equation*}
    \hat{\lambda}_x :x^{-1}xS \longrightarrow xS
\end{equation*}
is given by $(z)\hat{\lambda}_x = xz$, for all $z \in x^{-1}xS$. Moreover, $(S)\hat{\lambda}$ is left ample in $\mathcal{I}_S$. 
\end{theorem}

\noindent We now prove the main theorem of this section.

\begin{figure}[htbp]
\centering
\includegraphics [width=90mm]{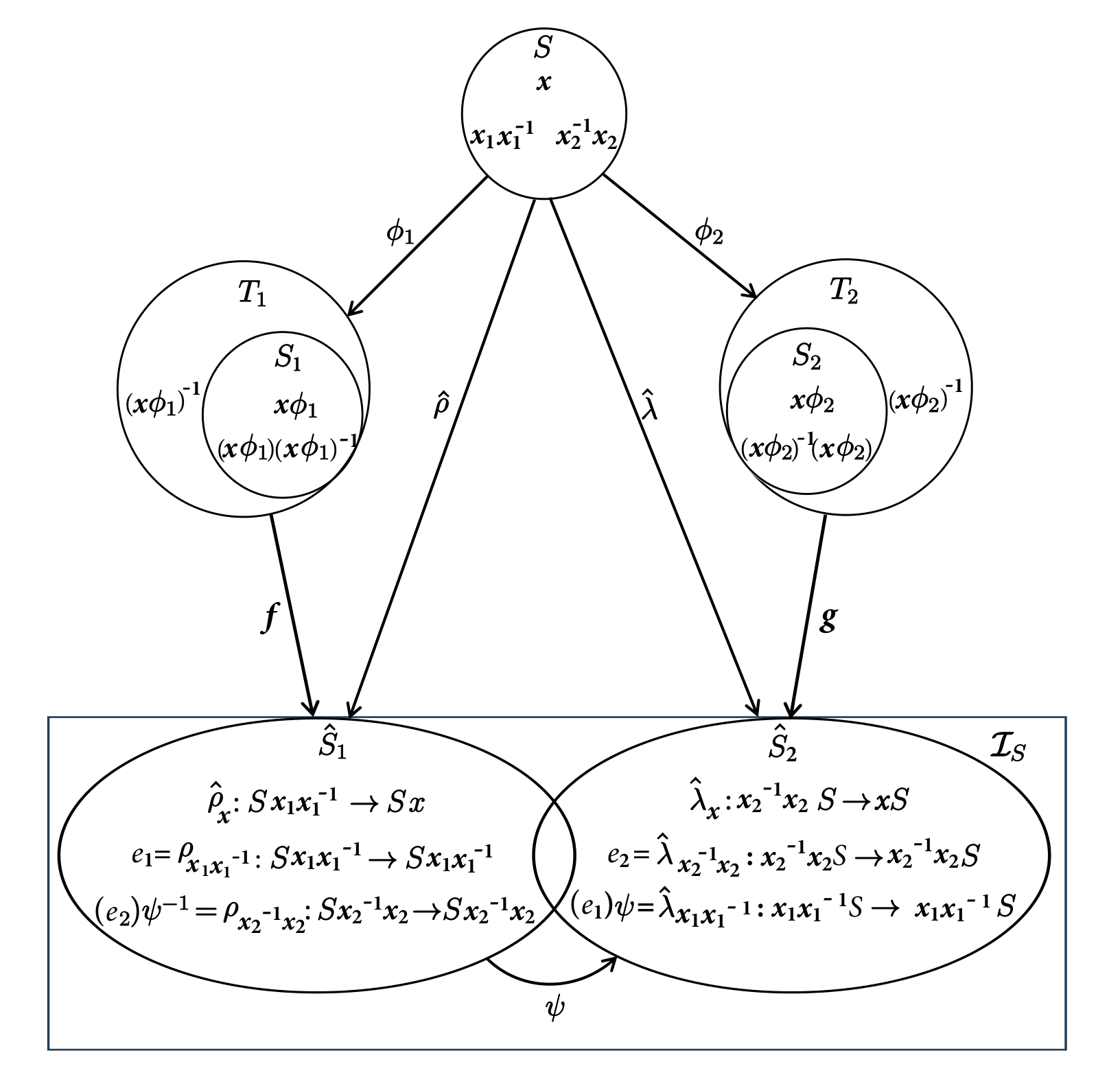}
\caption{}
\end{figure}\label{Diagram1}

\begin{theorem}\label{Thrm2.8}
Let $S$ be a finite ample semigroup such that for all $e\in E(S)$ there exists a bijection between the principal ideals $eS$ and $Se$. Then $S$ is ample in, and hence a $(2,1,1)$-subalgebra of, $\mathcal{I}_S$.
\end{theorem}

\begin{proof}
If $S$ is an inverse semigroup then there is nothing to prove. So, assume that $S$ is a finite non-inverse ample semigroup satisfying the condition of the theorem.\vspace{2pt}

Let $\phi_i:S\longrightarrow T_i,\, 1\leq i \leq 2$, be monomorphisms such that $S_1 = S\phi_1$ is left ample in $T_1$ and $S_2 = S\phi_2$ is right ample in $T_2$.
Also, for each $i \in \{1,2\}$, let $\phi'_i:S\longrightarrow S_i$ be the isomorphisms defined by $s \longmapsto x\phi_i$. In a way similar to (\ref{Identification}), define for all $x \in S$, $x_1x_1^{-1}, x_2^{-1}x_2 \in E(S)$ by
\begin{equation*}
\begin{split}
x_1 x_1^{-1}&:=[(x\phi_1)(x \phi_1)^{-1}]\phi_1'^{-1},\vspace{4pt} \\
x_2^{-1}x_2&:=[(x\phi_2)^{-1}(x \phi_2)]\phi_2'^{-1}.
\end{split}
\end{equation*}
Also, denote $S\hat{\rho}$ and $S\hat{\lambda}$ by $\hat{S_1}$ and $\hat{S_2}$, respectively, where $\hat{\rho}$ is the monomorphism given by Theorems \ref{Prop2.4} and $\hat{\lambda}$ is the anti-monomorphism given by Theorem \ref{Prop2.4a}. Then there exists an anti-isomorphism $\psi:\hat{S_1}\longrightarrow \hat{S_2}$ such that $\hat{\rho}\circ \psi = \hat{\lambda}$. Note that $f=\sigma \circ \chi: S_1 \longrightarrow \mathcal{I}_S$ is a monomorphism (refer to Figure 1). Also, there exists an anti-monomorphism  $g:S_2 \longrightarrow \mathcal{I}_S$ yielded from Theorem \ref{Prop2.4a}. (All these maps are shown in Figure 2.) \vspace{2pt}

\noindent Now, consider the following elements of $\hat{S_1}$:
\begin{equation*}
\begin{split}
    \hat{\rho}_x&:Sx_1x_1^{-1} \longrightarrow Sx,\\
    e_1 = \hat{\rho}_{x_1x_1^{-1}}&: Sx_1x_1^{-1} \longrightarrow Sx_1x_1^{-1}\\
    = \hspace{1.5pt}\hat{\rho}_{x} \hat{\rho}^{-1}_{x}&: Sx_1x_1^{-1} \longrightarrow Sx_1x_1^{-1}.
\end{split}
\end{equation*}
Analogously, the following elements belong to $\hat{S_2}$:
\begin{equation*}
\begin{split}
\hat{\lambda}_x&:x_2^{-1}x_2S \longrightarrow xS,\\
e_2 =\; \hat{\lambda}_{x_2^{-1}x_2}&: x_2^{-1}x_2S \longrightarrow x_2^{-1}x_2S\\
=\hspace{3.5pt} \hat{\lambda}_{x}\hat{\lambda}^{-1}_{x}&: x_2^{-1}x_2S \longrightarrow x_2^{-1}x_2S.
\end{split}
\end{equation*}
As $\psi(\hat{\rho}_x) = \hat{\lambda}_x$, we also note that
\begin{equation}\label{AntiIso}
\begin{split}
(e_2)\psi^{-1}=\hat{\rho}_{x_2^{-1}x_2}&:S{x_2^{-1}x_2} \longrightarrow S{x_2^{-1}x_2} \in \hat{S}_1,\\
    (e_1) \psi = \hat{\lambda}_{x_1x_1^{-1}}&: x_1x_1^{-1}S \longrightarrow x_1x_1^{-1}S \in \hat{S}_2.
\end{split}
\end{equation}
Let us now carry out the following calculation:
\begin{equation*}
\begin{split}
    (\hat{\rho}_x)(e_2\psi^{-1})&= (\hat{\rho}_x)\left((\hat{\lambda}_{x}\hat{\lambda}^{-1}_{x})\psi^{-1}\right)\\
    &=(\hat{\lambda}_x)\psi^{-1}\left((\hat{\lambda}_{x}\hat{\lambda}^{-1}_{x})\psi^{-1}\right)\\
    &= (\hat{\lambda}_{x}\hat{\lambda}^{-1}_{x} \hat{\lambda}_x)\psi^{-1}\\
    &=(\hat{\lambda}_x)\psi^{-1}\\
    &=\hat{\rho}_x
\end{split}
\end{equation*}

\noindent Because $\hat{\rho}_x^{-1}\hat{\rho}_x \in E(\mathcal{I}_S)$ is the minimal idempotent such that $\hat{\rho}_x(\hat{\rho}^{-1}_x\hat{\rho}_x)=\hat{\rho}_x$, we must have
\begin{equation}\label{Eqn2.6}
\hat{\rho}^{-1}_x\hat{\rho}_x \leq (e_2)\psi^{-1} =  \hat{\rho}_{x_2^{-1}x_2}, \hspace{8pt} \text{ refer to (\ref{AntiIso})}.
\end{equation}
By a similar token, we also obtain
\begin{equation}\label{Eqn2.7a}
{\hat{\lambda}_x^{-1}}{\hat{\lambda}_x} \leq\, (e_1)\psi = {\hat{\lambda}_{x_1x^{-1}_1}}.
\end{equation}
The proof will be accomplished if show that $\hat{\rho}^{-1}_x\hat{\rho}_x \in \hat{S}_1$, wherefore it suffices to show that
\[
Dom(\hat{\rho}^{-1}_x\hat{\rho}_x)  = Dom(\hat{\rho}_{x^{-1}_2x_2})\hspace{8pt}(=Dom(e_2)\psi^{-1}).
\]
Since $S$ is a finite semigroup and by (\ref{Eqn2.6}) we have $Dom(\hat{\rho}^{-1}_x\hat{\rho}_x) \subseteq Dom(\hat{\rho}_{x^{-1}_2x_2})$, the aim will be achieved if there exists a bijection between  $Dom(\hat{\rho}^{-1}_x\hat{\rho}_x)$ and $Dom(\hat{\rho}_{x^{-1}_2x_2})$. Let $\mid X \mid$ denote the cardinality of $X$. Then the aim is to show that
\[
\mid Dom(\hat{\rho}^{-1}_x\hat{\rho}_x) \mid \, =\, \mid Dom(\hat{\rho}_{x^{-1}_2x_2}) \mid.
\]
Using (\ref{Eqn2.6}) and (\ref{Eqn2.7a}) we first obtain,
\[
\mid Sx \mid \, =\, \mid Dom(\hat{\rho}^{-1}_x\hat{\rho}_x) \mid \, \leq \, \mid Dom(\hat{\rho}_{x_2^{-1}x_2}) \mid \, = \, \mid Sx_2^{-1}x_2 \mid,
\]
and
\[
 \mid x S \mid \, = \, \mid Dom({\hat{\lambda}_x^{-1}}{\hat{\lambda}_x}) \mid \, \leq \,\mid Dom ({\hat{\lambda}_{x_1x_1^{-1}}}) \mid \, = \, \mid x_1x_1^{-1}S \mid.
\]
Now, because $\mid x_2^{-1}x_2S \mid \, = \, \mid xS \mid \,$ and $\mid Sx_1x_1^{-1} \mid \, = \, \mid Sx \mid$, we obtain, by the hypothesis:
\begin{equation}\label{Eqn2.8}
 \, \mid Sx \mid \, \leq \, \mid Sx_2^{-1}x_2 \mid \, = \, \mid x_2^{-1}x_2S \mid \, = \, \mid xS \mid \, \leq \, \mid x_1x_1^{-1}S \mid \, = \, \mid Sx_1x_1^{-1}\mid \, = \, \mid Sx \mid.
\end{equation}
This implies that
\[
\mid Dom(\hat{\rho}^{-1}_x\hat{\rho}_x) \mid \, = \, \mid Sx\mid \, = \, \mid Sx^{-1}_2x_2 \mid \, = \, \mid Dom(\hat{\rho}_{x^{-1}_2x_2}) \mid.
\]
and the proof is complete.
\end{proof}

\begin{corollary}
    If $S$ is an ample semigroup then the following conditions are equivalent.
    \begin{enumerate}
        \item For every $e \in E(S)$ There exists a bijection between the principal ideals $eS$ and $Se$.
        \item For all $x \in S$ there exists a bijectoon between $xS$ and $Sx$.
    \end{enumerate} 
\end{corollary}
\begin{proof}
(1)$\implies$(2) It suffices to observe that in the proof of Theorem \ref{Thrm2.8} finiteness condition is not used to obtain (\ref{Eqn2.8}).\vspace{2pt}

\noindent(2)$\implies$(1) This part is trivial.\vspace{2pt}

\end{proof}

The class of all finite semigroups with central idempotents froms a pseudovariety, denoted by $\mathbf{ZE}$. Denoting by $\mathbf{G}$ and $\mathbf{Com}$ the pseudovarieties of groups and commutative semigroups, we have \cite{JA}:
\[
\mathbf{ZE} = \mathbf{G} \vee \mathbf{Com}.
\]

\noindent If we denote the class (in fact quasivariety) of ample semigroups by $\mathbf{A}$ then the following corollary shows that every member of $\mathbf{A} \wedge \mathbf{ZE}$ is embeddable in an inverse semigroup as $(2,1,1)$-subalgebra.

\begin{corollary}\label{cor:center-idempotents}
    Let $S$ be a finite ample semigroup such that $E(S)$ is contained in the center of $S$. Then $S$ is ample in (equivalently, a $(2,1,1)$-subalgebra of)  $\mathcal{I}_S$.
\end{corollary}

\begin{proof}
    Straightforward.
\end{proof}

\section{Examples concerning left (right) ample semigroups}

A semigroup $S$ without zero is called \textit{simple} if it has no proper (two-sided) ideals (recall that groups are precisely the semigroups with no proper one-sided ideals). We say that $S$ is \textit{completely simple} if it is simple and contains an idempotent which is minimal, within the set $E(S)$, with respect to $\preccurlyeq$; refer to (\ref{Eq1.0}). A semigroup $S$ with zero element $0$, is called $0$\textit{-simple} if $S^2 \neq {0}$ and ${0}$ is the only proper ideal of $S$. (The first condition only serves to exclude the two-element null semigroup, which makes the relevant structure theory rather cleaner.) If $S$ is a semigroup with zero then clearly $0 \preccurlyeq e$ for all $e \in E(S)$. In this case, an idempotent $f \in E(S)$ is called \textit{primitive} if $e \preccurlyeq f$ implies that $e \in \{0,f\}$ for all $e \in E(S)$. A semigroup $S$ with zero is said to be \textit{completely $0$-simple}
if it is $0$-simple and contains a primitive idempotent. Completely simple and $0$-simple semigroups form an important class of (regular) semigroups. They are characterized by the celebrated Rees representation theorem (reproduced below for the convenience of the reader).
In this section we shall consider completely $0$-simple inverse semigroups, known as \textit{Brandt semigroups}. (Completely simple inverse semigroups are in fact groups.)\vspace{2pt}

Let $G^0$ be the semigroup obtained by externally adjoining a zero element $0$ to a group $G$ (we call $G^0$ a \textit{zero group}).  Let $I$ and $\Lambda$ be some non-empty (indexing) sets and $P =(p_{\lambda i})$ be a $\Lambda \times I$ matrix over $G^0$ such that none of its rows or columns consists entirely of zeros (in the literature such a matrix is called \textit{regular}). We use $P$ to define an associative binary operation on the set $
(I \times G \times \Lambda) \cup \{0\}$, as follows:
\[   
(i,a,\lambda)(j,b,\mu) = 
     \begin{cases}
       (i, ap_{\lambda j}b, \mu), &\quad\text{if } p_{\lambda j}\neq 0,\\
       0, &\quad\text{if } p_{\lambda j}=0, \\
     \end{cases}
\]
\[
(i,a,\lambda)0=0(i,a,\lambda)=00=0.  
\]
The set $(I \times G \times \Lambda) \cup \{0\}$ together with the above binary operation is called the $I \times \Lambda$ \textit{Rees matrix semigroup} over the $0$-group $G^0$ with \textit{sandwich matrix} $P$. We denote it by $\mathcal{M}^0(G^0, I, \Lambda, P)$. For further background on Rees matrix semigroups the reader may refer to \cite{CP}, Chapter 3.

\begin{theorem}[Rees representation theorem; \cite{Howie}, Theorem 3.2.3]\label{Rees}
    Every completely $0$-simple semigroup is isomorphic to some Rees matrix semigroup $\mathcal{M}(G^{0}, I, \Lambda, P)$. 
\end{theorem}

\noindent To characterize Brandt semigroups as a subclass of Rees matrix semigroups, we consider an $I \times I$ sandwich matrix $\Delta$ that has $1$s on the diagonal and zeros elsewhere. We then have the following theorem.
\begin{theorem}[\cite{Howie}, Theorem 5.1.8]
    Every Brandt semigroup is isomorphic to a Rees matrix semigroup $\mathcal{M}^0(G^0,I,I,\Delta)$.
\end{theorem}

\noindent Since the sandwich matrix $\Delta$ will always be obvious, the Brandt semigroup $ \mathcal{M}^0(G^0,I,I,\Delta)$ will, from now on, be denoted by its universe $(I \times G \times I) \cup \{0\}=B$.  Note that the inverse of $(i,g,j) \in B$ is $(j,g^{-1},i)$, where $g^{-1}$ is the inverse of $g$ in $G$. Also, the non-zero idempotents of $B$ are of the form $(i,1,i),\, i \in I$, where $1$ is the identity of $G$.\vspace{2pt}

It is an easy exercise to verify that a semigroup with zero that has no proper non-zero one sided ideal is a zero group (and hence an inverse semigroup). This implies that commutative $0$-simple semigroups, equivalently, the commutative Brandt semigroups, are zero (Abelian) groups. It is also a routine to verify that a Brandt semigroup is a zero group if and only if $I$ is a singleton. We shall consider left (right) ample subsemigroups of a Brandt semigroup, such that the latter is not a zero group. So, we shall only be dealing with non-commutative Brandt semigroups such that the set $I$ contains more than one element.
\begin{definitions}
By a \textit{strict left ample} semigroup we mean a semigroup that is left but not right ample. \textit{Strict right ample} semigroups are defined dually.
\end{definitions}

\begin{remark}\label{Prop3.4}
A semigroup $S$ is a strict left ample if and only if the following conditions hold.
\begin{enumerate}
    \item $(S)\hat{\rho} = \hat{S}_1$ is left ample in $\mathcal{I}_S$, refer to Theorem \ref{Prop2.4} and Figure 2. \vspace{4pt}
    \item  Either $\hat{\lambda}$ is not an anti-monomorphism or $(S)\hat{\lambda} = \hat{S}_2$ is not left ample in $\mathcal{I}_S$, cf. Theorem \ref{Prop2.4a} and Figure 2.
\end{enumerate}
Dual conditions hold for strict right ample semigroups.
\end{remark}

\begin{proof}
($\implies$) Let $S$ be strict left ample. Then, by Theorem \ref{Prop2.4}, $\hat{S}_1$ is left ample in $\mathcal{I}_S$. To prove condition (2), assume on the contrary that $\hat{\lambda}$ is an anti-monomorohism and $\hat{S}_2$ is left ample in $\mathcal{I}_S$. Consider
\[
\hat{S}^{-1}_2 = \{x^{-1} : x \in \hat{S}_2\},
\]
a subsemigroup of $\mathcal{I}_S$. Then, $h:\hat{S}_2\longrightarrow \hat{S}^{-1}_2$, given by $x \mapsto x^{-1}$, is clearly an anti-isomorphism, and $\hat{S}^{-1}_2$ is right ample in $\mathcal{I}_S$. But then $\hat{\lambda} \circ h : S \longrightarrow \mathcal{I}_S$ is a monomorphism, implying that $S$ is right ample, a contradiction.\vspace{2pt}

\noindent ($\impliedby$) By condition (1) $S$ is left ample. Now, suppose on the contrary that there exists a monomorphism $\phi:S \longrightarrow T$ such that $(S)\phi$ is right ample in $T$. Then by Theorem \ref{Prop2.4a}, $\hat{\lambda}:S \longrightarrow \mathcal{I}_S$ is an anti-monomorohism and $(S)\hat{\lambda} = \hat{S}_2$ is left ample in $\mathcal{I}_S$, a contradiction.
The dual conditions for strict right ample semigroups can be verified similarly.
\end{proof}

\begin{theorem}\label{Brandt}
Let $B$ be a Brandt semigroup that is not a zero group. Then, for every $e \in E(B) \smallsetminus {0}$, the principle ideals $eB$ and $Be$ are, respectively, strict left and strict right ample subsemigroups of $B$.
\end{theorem}

\begin{proof}
Let us arbitrarily fix an element $e=(\lambda,1,\lambda) \in E(B) \smallsetminus \{0\}$. We shall only show that $eB$ is a strict left ample subsemigroup of $B$ (the statement about $Be$ can be proved by a dual argument). 
Let $u$ be an arbitrary element of $eB$. If $u = 0$, then $uu^{-1} =0 \in eB$ and there is nothing to prove.  So, assume that $u\in eB \smallsetminus \{0\}$. Then, we have
\[
u=(\lambda,1,\lambda)(j,h,l) = (\lambda, h, l),\, \text{ where } \,(j,h,l)\in B \, \text{ with } j = \lambda.
\]
 Now, one can easily verify that $ uu^{-1}=(\lambda,1,\lambda)$, where $u^{-1}$ denotes the inverse of $u$ in $B$. 
Clearly, $(\lambda,1,\lambda)$ is an element of $eB$. Hence, $eB$ is left ample in $B$. \vspace{2pt}
  
Next, we use Remark \ref{Prop3.4} to prove that $S =eB$ is not right ample. The aim is to show that $\hat{\lambda}:S \longrightarrow \mathcal{I}_S$ is not an anti-monomorphism. Consider, for this purpose, an element $u = (\lambda, h, l) \in eB \smallsetminus \{0\}$. Let $u^{-1}$ denote the inverse of $u$ in some inverse oversemigroup $T$ of $eB$, such that $u^{-1}u \in eB$ (if no such inverse semigroup exists then we are done). Then, after effectuating necessary identifications, one can easily verify from $uu^{-1}u = u \neq 0$ that $u^{-1}u=(l,1,l)$. Because $I$ is not a singleton, we can choose $u$ such that $l \neq \lambda$. But then,
\begin{equation*}
\hat{\lambda}_u :(l,1,l)(eB) \longrightarrow (\lambda,h, l)(eB)
\end{equation*}
coincides with 
\begin{equation*}
\hat{\lambda}_0:0B \longrightarrow 0B,
\end{equation*}
because
\begin{equation*}
\begin{split}
 (l,1,l)(eB)=(l,1,l)(\lambda,1,\lambda)B  = 0B, \vspace{4pt}\\ (\lambda,h, l)eB = (\lambda,h, l)(\lambda,1,\lambda)B = 0B,
\end{split}
\end{equation*}
implying the $\hat{\lambda}$ is not an anti-monomorphism. This implies, by Proposition \ref{Prop3.4}, that $S =eB$ is not right ample.
\end{proof}

The following example provides a class of strict left ample subsemigroups of Brandt semigroups that are not contained in any non-zero idempotent generated ideals. 

\begin{example}
Consider a Brandt semigroup $B = (I \times G \times I) \cup \{0\}$, where $G$ is group and $I = \{1,2,3\}$.  Let $S = (\{1,2\} \times H \times I) \cup \{0\}$, where $H$ is a submonoid of $G$. Then we have, for all $g \in G$, $j \in \{1,2\}$ and $i \in I$, $(j,g,i)(i,g^{-1},j) = (j,1,j) \in S$. So, identifying $S$ with $S\hat{\rho}$, we see that Condition (1) of Remark \ref{Prop3.4} is satisfied. To prove Condition (2), we identify $S$ with $S\hat{\lambda}$ and notice that $(3,g^{-1},j)(j,g,3) = (3,1,3) \not\in S$. Hence, by Remark \ref{Prop3.4}, $S$ is strict left ample. At the same time, all principal ideals of $B$ are of the form $(\{i\} \times G \times I) \cup \{0\}$ and $S \not\subset (\{i\} \times G \times I) \cup \{0\}$ for any $i \in I$.
\end{example}

Let $B=(I\times G\times I)\cup\{0\}$ be a Brandt semigroup. Then clearly every subsemigroup of $B$ with zero is of the form $S = (I'\times H\times I'')\cup\{0\}$ where $I'$ and $I''$ are subsets of $I$ and $H$ is a subsemigroup of the group $G$. Obviously, a subsemigroup of $B$ without zero is of the form $I'\times H\times I''$.

\begin{theorem}
    Let $S = (I' \times G \times I'') \cup \{0\}$ be a non-trivial subsemigroup of a Brandt semigroup $B = (I \times G \times I) \cup \{0\}$. Then $S$ is left (respectively, right) ample in $B$ if and only if $H$ is a left (respectively, right) ample submonoid of the group $G$.
\end{theorem}

\begin{proof}
If $S$ is left ample in $B$ then for any non-zero element $s=(\lambda,a,\mu) \in S$ we have $ss^{-1}=(\lambda,a,\mu) (\mu,a^{-1},\lambda) =(\lambda,1,\lambda)\in S$. This implies that $aa^{-1}=1\in H$ for all $a \in H$, whence $H$ is a left ample submonoid of the group $G$.\vspace{2pt}

Conversely if $H$ is a left ample submonoid of the group $G$ then for every $a\in H$ and its inverse $a^{-1}\in G$, $aa^{-1}\in H$. So $ss^{-1}=(\lambda,a,\mu)(\mu,a^{-1},\lambda)=(\lambda,aa^{-1},\lambda)\in S$, that is $S$ is left ample. A similar argument applies if \lq left ample' is replaced by \lq right ample'.    
\end{proof}

\textbf{Acknowledgment.}
\upshape The authors are thankful to Professor Valdis Laan and dr. Ülo Reimaa for their valuable comments on the earlier drafts of this article. This research has been supported by Estonian Science Foundation's grants PRG1204.

\end{document}